\documentclass[letterpaper, 10 pt, conference]{ieeeconf}  
\IEEEoverridecommandlockouts                              
\usepackage{cite}
\usepackage{amsmath,amssymb,amsfonts}
\usepackage{algorithmic}
\usepackage{graphicx}
\usepackage{textcomp}
\usepackage{setspace}
\usepackage{booktabs}

\usepackage{epsfig}
\usepackage{times} 
\usepackage{svg}
\usepackage[linesnumbered,ruled,vlined]{algorithm2e} 

\usepackage[nolist, nohyperlinks]{acronym}

\usepackage{bm}
\usepackage{breqn}
\usepackage{amsmath}
\usepackage{amssymb}
\usepackage{tabularx}

\usepackage{amsthm}

\newcommand{\reals}{\mathbb{R}}
\newcommand{\R}{\reals}
\newcommand{\Rnonneg}{\reals_{\geq 0}}

\newcommand{\toSet}{\rightrightarrows}

\newcommand{\Bcal}{\mathcal{B}}

\newcommand{\Ical}{\mathcal{I}}
\newcommand{\Jcal}{\mathcal{J}}

\newcommand{\Ucal}{\mathcal{U}}

\newcommand{\Xcal}{\mathcal{X}}

\newcommand{\eqn}[1]{\begin{align} #1 \end{align}}
\newcommand{\eqnN}[1]{\begin{align*} #1 \end{align*}}
\newcommand{\seqn}[2][]{
\begin{subequations}
    #1
\begin{align} #2 \end{align}
\end{subequations}
}
\newcommand{\bmat}[1]{\begin{bmatrix}#1\end{bmatrix}}

\newcommand{\norm}[1]{\left\Vert #1 \right \Vert}

\newcommand{\abs}[1]{\left | #1 \right |}

\newcommand{\argmin}[1]{\operatorname*{argmin}_{\substack{#1}}}
\newcommand{\st}{\operatorname*{s.t.}}

\theoremstyle{plain}
\newtheorem{theorem}{Theorem}
\newtheorem{conjecture}{Conjecture}
\newtheorem{corollary}{Corollary}
\newtheorem{lemma}{Lemma}
\newtheorem{problem}{Problem}
\newtheorem{definition}{Definition}

\newtheorem{assumption}{Assumption}

\theoremstyle{definition}
\newtheorem{remark}{Remark}

\theoremstyle{remark}

\newtheorem{example}{Example}

\makeatletter
\let\NAT@parse\undefined
\makeatother
\usepackage{hyperref}
\hypersetup{
colorlinks, 
    linkcolor={red!50!black},
    citecolor={blue!50!black},
    urlcolor={blue!80!black}
}
\usepackage{cleveref}

\crefformat{problem}{Problem~#2#1#3}
\crefformat{assumption}{Assumption~#2#1#3}
\crefformat{property}{Property~#2#1#3}
\crefformat{prop}{Proposition~#2#1#3}
\crefformat{remark}{Remark~#2#1#3}
\crefformat{conjecture}{Conjecture~#2#1#3}

\usepackage[nolist,nohyperlinks]{acronym}
\acrodef{QP}[QP]{Quadratic Program}
\acrodef{QCQP}[QCQP]{Quadratically Constrained Quadratic Program}
\acrodef{SOCP}[SOCP]{Second-Order Cone Problem}
\acrodef{MFCQ}[MFCQ]{Mangasarian-Fromovitz Constraint Qualification}
\acrodef{SSOC}[SSOC]{Strong Second Order Condition}
\acrodef{LICQ}[LICQ]{Linear Independence Constraint Qualification}
\acrodef{KKT}[KKT]{Karush–Kuhn–Tucker}
\acrodef{SC}[SC]{Slater's Condition}
\acrodef{CLF}[CLF]{Control Lyapunov Function}
\acrodef{CBF}[CBF]{Control Barrier Function}
\acrodef{ODE}[ODE]{Ordinary Differential Equation}

\title{\LARGE \bf
Reformulations of Quadratic Programs for Lipschitz Continuity
}

\author{Devansh R. Agrawal$^{\dagger}$, Haejoon Lee$^{\dagger}$, and Dimitra Panagou
\thanks{This work is supported by the Air Force Office of Scientific Research (AFOSR) under FA9550-23-1-0163 and the National Science Foundation (NSF) under Award Number 1942907.}
\thanks{$^\dagger$ Both D. R. Agrawal and H. Lee have equal contribution.}
\thanks{All authors are with the Robotics Department, University of Michigan, Ann Arbor, MI, USA
        {\tt \{devansh, haejoonl, dpanagou\}@umich.edu}}
}

\begin{document}

\maketitle
\thispagestyle{empty}
\pagestyle{empty}

\begin{abstract}
Optimization-based controllers often lack regularity guarantees, such as Lipschitz continuity, when multiple constraints are present. When used to control a dynamical system, these conditions are essential to ensure the existence and uniqueness of the system's trajectory. Here we propose a general method to convert a \acf{QP} into a \acf{SOCP}, which is shown to be Lipschitz continuous. Key features of our approach are that (i)~the regularity of the resulting formulation does not depend on the structural properties of the constraints, such as the linear independence of their gradients; and (ii)~it admits a closed-form solution, which is not available for general \acp{QP} with multiple constraints, enabling faster computation. We support our method with rigorous analysis and examples. \href{https://github.com/joonlee16/Lipschitz-controllers}{[Code]}\footnote{Code: https://github.com/joonlee16/Lipschitz-controllers}
\end{abstract}


\section{Introduction}
\label{sec:intro}

This paper investigates the Lipschitz continuity of parametric \acfp{QP} of the form
\seqn[\label{eqn:intro_qp}]{
 \min_{u \in \R^m} \ &  \norm{ u - \pi_{\rm des}(x)}^2\\
\st \ & u \in K(x),
}
where $x \in \Xcal \subset \R^n$ is a parameter, and
\eqn{
K(x) = \left \{ u \in \R^m: A(x) u \leq b(x) \right \} \label{eqn:K}
}
is a convex polyhedral set with $A: \Xcal \to \R^{p \times m}$, $b: \Xcal \to \R^p$, and $\pi_{\rm des}: \Xcal \to \R^m$. We assume throughout that $K(x)$ is nonempty at each $x \in \Xcal$.  The set of minimizers of~\eqref{eqn:intro_qp} is denoted by
\eqn{
\Pi(x) = \argmin{u \in K(x)} \norm{ u - \pi_{\rm des}(x)}^2, \label{eqn:main_qp}
}
where $\Pi: \Xcal \toSet \R^m$ is a set-valued map. Since $K$ is always nonempty and the objective function is strongly convex, $\Pi(x)$ contains exactly one element at each $x \in \Xcal$. This unique minimizer is denoted by $\pi: \Xcal \to \R^m$. 

Lipschitz continuity of $\pi$, the (unique) minimizer of parametric problems such as~\eqref{eqn:main_qp}, is of fundamental importance. For example, consider a nonlinear dynamical system 
\eqn{
\dot x = f(x, u) = f(x, \pi(x)) = f_{\rm cl}(x), \label{eqn:dynamics}
}
where the control input determined by~\eqref{eqn:main_qp}, that is, $u = \pi(x)$. When $\pi$ is Lipschitz continuous, its sensitivity to changes in $x$ is bounded. That is, small perturbations in the state (induced by noise, numerical inaccuracies, etc.) result in bounded deviations in the control input, thereby preventing instability or chattering behavior of the system. Moreover, the Lipschitz continuity of $\pi$ induces Lipschitz continuity of the closed-loop dynamics $f_{\rm{cl}}$. In turn, this ensures that the system~\eqref{eqn:dynamics} admits a unique solution for all time~\cite{khalil2002nonlinear}, guaranteeing that the system trajectories are unique.

\begin{figure}[t]
    \centering
    \includegraphics[width=0.9\linewidth]{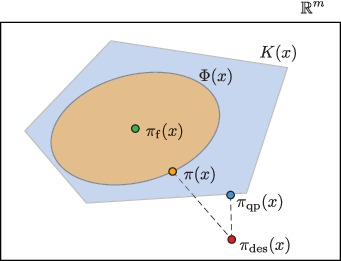}
    \caption{A depiction of the main problem and our proposed solution approach. In parametric \acp{QP}, the solution $\pi_{\rm qp}(x)$ may be non-Lipschitz or even discontinuous with respect to $x$. This paper defines a subset $\Phi(x)$ of the constraint set $K(x)$ that is always feasible and Lipschitz continuous on $\R^m$ (defined in~\Cref{assum:main}) such that the modified solution $\pi(x)$ is Lipschitz continuous. This is done by identifying a Lipschitz continuous function $\pi_{\rm f}(x)$ that is always feasible, and constructing a set $\Phi(x)$ around it such that $\pi_{\rm socp}(x)$ is Lipschitz wrt $x$.}
    \label{fig:problem}
\end{figure}


However, it is well known that in the absence of certain constraint qualifications (e.g.,~\ac{LICQ}), $\pi$ can be non-Lipschitz or even discontinuous~\cite{robinson1982, liu_sensitivity_1995, alyaseen2025continuity}. This is true even when $A, b, \pi_{\rm des}$ are continuously differentiable with respect to $x$ (see Example~\ref{ex:dev_ex}), or \acf{SC} is satisfied (see \Cref{ex:motivation_ex}). 
\begin{example}
\label{ex:dev_ex}
 Let $x \in \Xcal = \R$. Consider the \ac{QP} 
\seqn[\label{eq:dev_ex_eq}]{
\label{eq:ex1}
\pi_{\rm qp}(x) = \argmin{u \in \R^2} \ & \norm{ u - \bmat{-2 \\ 0 }}^2\\
\st \ & \bmat{1 & 0 \\ -1 & -x} u \leq \bmat{1 \\ -(1+x)}.
}
The solution of~\eqref{eq:dev_ex_eq} can be derived analytically, and is plotted in~\Cref{fig:example_1}~(c). Even though $A, b$ are continuously differentiable, notice that there exists a discontinuity of solution at $x=0$.

\end{example}

\begin{figure*}[t]
    \centering
    \includegraphics[width=\linewidth]{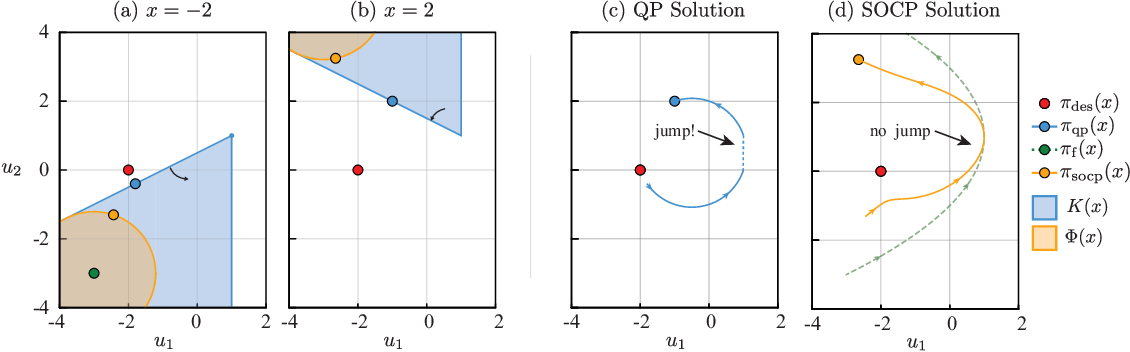}
    \caption{Fig.~(a) and~(b) depict the feasible spaces and solutions of the original \ac{QP}~\eqref{eq:dev_ex_eq} and of our proposed \ac{SOCP} reformulation in the form of~\eqref{eqn:socp_r} at $x=-2$ and $x=2$, respectively. Note at $x = 0$, $K(x)$ shrinks to $\{ u \in \R^2 : u_1 = 1 \}$, a set with no interior. This causes $\pi_{\rm qp}$ to jump, as depicted in Fig.~(c). In contrast, the \ac{SOCP} reformulation introduces a smoothly varying feasible space $\Phi(x)$ centered at $\pi_{\rm f}(x)$, ensuring that the solution $\pi_{\rm socp}(x)$ transitions without jumps, as depicted in Fig.~(d). Animations are available in our \href{https://github.com/joonlee16/Lipschitz-controllers}{supplementary repository}.}
    \label{fig:example_1}
\end{figure*}



Since discontinuous solutions may not be preferable in practice, e.g., due to chattering phenomena, the focus of this paper is to reformulate~\eqref{eqn:intro_qp} into a problem of the form
\seqn[\label{eqn:intro_socp}]{
\min_{u \in \R^m} \ & \norm{ u - \pi_{\rm des}(x)}^2\\
\st \ &  u \in \Phi(x) \subset K(x), 
} 
whose minimizer is both unique and Lipschitz continuous with respect to $x$. To this end, we propose a suitable definition for $\Phi: \Xcal \toSet \R^m$ and demonstrate that, by construction, its properties ensure the Lipschitz continuity of the minimizer of~\eqref{eqn:intro_socp}.
\Cref{fig:problem} depicts this pictorially.

\emph{Literature Review:} Parametric optimization --- both convex and nonconvex --- has been extensively studied~\cite{fiacco1983introduction, bonnans2013perturbation, rockafellar_variational_2009, still2018lectures}, with a central focus on establishing conditions under which the solution map exhibits regularity~\cite{morris15, hager_lipschitz_1979, mestres_regularity_2025, liu_sensitivity_1995, robinson1982}. Under constraint qualifications such as \ac{LICQ}, \ac{MFCQ}, and \ac{SC}, solutions to \acp{QP} can be differentiable, Lipschitz, or H{\"o}lder continuous, depending on the setting. See~\cite{still2018lectures, mestres_regularity_2025} for detailed summary.

To apply these results, one must verify constraint qualifications hold. Although feasible in some cases~\cite{xu2015robustness, agrawal2022safe, wu_quadratic_2023}, this is generally difficult. One alterative approach is to combine multiple constraints into a unified constraint~\cite{black2023adaptation, egerstedt2018, breeden_compositions_2023}, where \ac{LICQ} holds trivially. However, this may render the problem infeasible even if $K(x)$ is non-empty.

Another line of work analyzes regularity of the feasible set map $K:\Xcal \toSet \R^m$, treating it as a set-valued mapping~\cite{freeman2008robust, rockafellar_variational_2009, bednarczuk2021lipschitz}. This literature studies whether $K$ is Lipschitz continuous and when Lipschitz selections of $K$ exist.

\emph{Contributions:} In this paper, we take a fundamentally different approach. Instead of focusing on the structural conditions of the~\ac{QP} like constraint qualifications, we
\begin{itemize}
    \item propose a method to reformulate the \ac{QP}~\eqref{eqn:main_qp} into a \acf{SOCP}~\eqref{eqn:intro_socp} that is guaranteed to admit a unique Lipschitz continuous solution under mild assumptions,
    \item show that our \ac{SOCP} admits a closed-form solution, 
    \item illustrate the challenges and ideas through examples,
    \item provide a conjecture to expand the set of feasible solutions and reduce conservatism.
\end{itemize}


The core theoretical insight is that by carefully tightening the constraint set, we ensure Lipschitz continuity while preserving uniqueness. A key source of non-Lipschitz behavior in~\eqref{eqn:main_qp} is when two active constraints become linearly dependent, as shown in~\Cref{ex:dev_ex}. Small perturbations in this case can drastically alter the feasible set and solution. Our formulation defines $\Phi$ to avoid such degeneracies.

\section{PRELIMINARIES}
\label{sec:prelim}

Let $\R$ denote the set of reals. Let $C^k$ denote the set of functions with $k$ continuous derivatives.




\begin{definition}[{\cite[Def 9.1]{still2018lectures}}]
    Let $X, Y$ be normed spaces. A function $f: X \to Y$ is \textbf{Lipschitz continuous on $X$} if there exists a $L \geq 0$ such that
            \eqn{
            \norm{f(x_1) - f(x_2)}\leq L \norm{x_1 - x_2}, \ \forall x_1, x_2 \in X.
            }
    $f$ is \textbf{locally Lipschitz at $x \in X$} if there exists a neighborhood $\Bcal \subset X$ of $x$ and a $L = L(x) \geq 0$ such that
    \eqn{
    \norm{f(x_1) - f(x_2)}\leq  L \norm{x_1 - x_2}, \ \forall x_1, x_2 \in \Bcal.
    }
    $f$ is locally Lipschitz on $\Xcal$ if it is locally Lipschitz $\forall x \in X$.
\end{definition}
 A Lipschitz function is also locally Lipschitz, while the converse is not necessarily true~\cite[Th. 2.1.6]{lipschitz_functions_textbook}. Lipschitz continuity can also be extended to set-valued maps.
\begin{definition}[{\cite[Def. 9.26]{rockafellar_variational_2009}}]
\label{def:set_lipschitz}
    A set-valued map $K: \Xcal \toSet \R^m$ is \textbf{Lipschitz continuous on $\Xcal \subset \R^n$} if it is nonempty-closed-valued on $\Xcal$ and there exists a $L \geq 0$ such that 
    \eqn{
    d_{H}(K(x_1), K(x_2)) \leq L \norm{x_1 - x_2} \ \forall x_1, x_2 \in \Xcal.
    }
    where $d_H$ is the Hausdorff distance~\cite[Def. 4.13]{rockafellar_variational_2009}.
\end{definition}

\subsection*{Parametric Convex Optimization}

Recall problem~\eqref{eqn:intro_qp}. Let $\Jcal=\{1,\dots, p\}$ denote the constraint indices. Let $a_i(x)$ and $b_i(x)$ denote the $i^{\text{th}}$ row of $A(x)$ and $b(x)$ respectively, i.e., $a_i : \Xcal \to \R^m$, $b_i: \Xcal \to \R$.
The $i^{\rm th}$ constraint is active at $x$ \textit{iff} $b_i(x) = a_i(x)^\top\pi(x)$.
A set of constraint with indices $\Ical \subset \Jcal$ are linearly dependent at $x \in \Xcal$ if vectors $a_i(x)$, $i \in \Ical$, are linearly dependent.

Various constraint qualifications have been studied in the literature, see~\cite{bonnans2013perturbation,rockafellar_variational_2009,mestres_regularity_2025}. Of particular interest is \ac{SC}:

\begin{definition}[Section 2.2~\cite{mestres_regularity_2025}]
\label{def:slaters}
    For~\eqref{eqn:intro_qp}, \textbf{\acf{SC}} holds at $x \in \Xcal$ if there exists $u \in \Ucal$ such that $a_i(x)^\top u < b_i(x)$, for all $i \in \Jcal$.
\end{definition}

Although \ac{SC} guarantees existence of a smooth selection~\cite[Proposition~3.1]{pio2019}, finding this function remains an open problem. Under certain conditions, \ac{SC} guarantees continuity (not local Lipschitz) of its solution mapping~\cite{morris15, mestres_regularity_2025}, for example in scalar optimization problems~\cite[Proposition 3.3]{mestres_regularity_2025}.  This does not extend to higher-dimensional cases. Lipschitz continuity of $K$ (in the sense of~\Cref{def:set_lipschitz}) is also not sufficient for solutions of~\eqref{eqn:intro_qp} to be Lipschitz --- solutions are only H{\"o}lder continuous~\cite{bednarczuk2021lipschitz}, \cite[Ex. 1.2]{dontchev1983perturbations}.



The next example demonstrates a more subtle phenomenon, where even when the original problem satisfies strong regularity conditions such as \ac{SC}, its minimizer may still fail to be Lipschitz.

\begin{example}
\label{ex:motivation_ex} Consider the following problem inspired by~\cite{robinson1982}:
\seqn[\label{eq:motivation_ex_eq}]{
\label{eq:mot_ex}
\argmin{u \in \R^2 } \ & \norm{ u}^2\\
\st \ & \bmat{-1 & 0 \\ -1 & -x_1 } u \leq \bmat{-1 \\ -(1+x_2)}
}
for $x =\bmat{x_1 & x_2}^\top\in \Xcal = \{x\in \R^2 : \norm{x}_\infty\leq 2\}$. The problem admits a strictly feasible solution 
\eqn{
\label{eq:feasible_solution_ex2}
\pi_{\rm f}(x)=\bmat{2+\abs{x_2} & 0}^\top,
}
and therefore \ac{SC} is satisfied. 
However, for  $x_2 \leq \frac{1}{2} x_1^2$, the solution is 
\eqn{\pi_{\rm qp}(x) = \begin{cases}
    \bmat{1 & 0 }^\top & \text{if } x_2 \leq 0,\\
    \bmat{1 & \frac {x_2}{x_1}}^\top & \text{if } 0 < x_2 \leq  \frac{1}{2} x_1^2,
\end{cases}
}
which implies that~\eqref{eq:motivation_ex_eq} is not Lipschitz at $\bar x = \bmat{ 0 &  0}^\top$.
\end{example}

\subsection*{Problem Statement}
Examples~\ref{ex:dev_ex} and~\ref{ex:motivation_ex} illustrate that \acs{QP} may not be (locally) Lipschitz continuous. These discontinuities stem from the unbounded sensitivity of the solutions about the point where their active constraints have linearly dependent rows in $A(x)$. Small perturbations in $x$ can cause abrupt changes in the active constraint set, leading to sharp, potentially unbounded changes in the solution along the active affine constraints. Other similar examples can be found in \cite{liu_sensitivity_1995, mestres_regularity_2025}. Thus, the goal of this paper is the following:
\begin{problem}
    Given an optimization problem~\eqref{eqn:intro_qp}, design a $\Phi: \Xcal \toSet \R^m$ such that $\Phi(x) \subset K(x)$ for all $x \in \Xcal$ and the modified optimization problem~\eqref{eqn:intro_socp} 
    admits a unique solution $\pi: \Xcal \to \Ucal$ for every $x \in \Xcal$ and such that  $\pi$ is Lipschitz.
    \label{prob:problem}
\end{problem}


We make the following mild assumptions: 
\begin{assumption}~
    \begin{enumerate}
        \item Each row of $A:\Xcal \to \R^{p\times m}$ has unit-norm, i.e., $\norm{a_i(x)}=1$, $\forall i \in \Jcal$. 
        \item The functions $a_i: \Xcal \to \R^{m}$, $b_i: \Xcal \to \R$, and $\pi_{\rm{des}}: \Xcal \to \R^m$ are Lipschitz continuous, with constants $L_{a_i}, L_{b_i}, L_{\pi_{\rm{des}}} \geq 0$ respectively. 
    \end{enumerate}
    \label{assum:Ab}
\end{assumption}


\section{SOLUTION}
\label{sec:main}

\subsection{Main Result}

Our key idea is that if we know that there exists a $\pi_{\rm f}: \Xcal \to \Ucal$ that is Lipschitz on $\Xcal$ and satisfies the constraints, we can exploit this solution to construct a solution to~\cref{prob:problem}. We formalize this as follows:
\begin{assumption}~
There exists $\pi_{\rm f}: \Xcal \to \Ucal$ such that 
        \eqn{
        \pi_{\rm f}(x) \in K(x) \quad \forall x \in \Xcal,
        }
        $\pi_{\rm f}$ is Lipschitz on $\Xcal$, and $\pi_{\rm f}$ is bounded: $\exists \overline U_{\rm f} >0$ such that $\norm{ \pi_{\rm f}(x)} \leq \overline{U}_{\rm f}$ $\forall x \in \Xcal$.
    \label{assum:main}
\end{assumption} 


\begin{remark}
   Note that $\pi_{\rm f}$ directly addresses~\Cref{prob:problem}: $\Phi(x) = \{ \pi_{\rm f}(x) \}$ is a valid solution to~\Cref{prob:problem}. However, $\pi_{\rm f}$ does not consider the objective of the original \ac{QP} and thus can lie far from the desired solution. As shown in~\Cref{fig:problem}, the new formulation~\eqref{eqn:socp_r} leverages $\pi_{\rm f}$ to reduce this conservatism, while preserving Lipschitz continuity. While such a $\pi_{\rm f}$ might not be easily available for general problems, we propose several methods to construct $\pi_{\rm f}$ and thus relax the assumption later in~\Cref{sec:pi_f}.
\end{remark}

We propose the following definition for $\Phi$:
\eqn{
\Phi(x) = \{ u \in \R^m : g_i(x, u) \leq 0, \forall i \in \Jcal\},
}
where
\eqn{
\label{eqn:g_i}
g_i(x, u) = \norm{ u - \pi_{\rm f}(x)} - (b_i - a_i(x)^\top  \pi_{\rm f}(x)).
}

This corresponds to the following modified problem:
\seqn[\label{eqn:socp_r}]{
\pi_{\rm socp}(x) = \argmin{u \in \R^m} \ &   \norm{ u - \pi_{\rm{des}}(x)}^2 \\
\st \ & \norm{ u - \pi_{\rm f}(x)} \leq r(x)\label{eqn:socp_cons}
}
where 
\eqn{
r(x) = \min_{i \in \Jcal} \ ( b_i(x) - a_i(x)^\top \pi_{\rm f}(x) ).
}


The $K(x)$ and $\Phi(x)$ sets are depicted in~\Cref{fig:example_1}. The blue region shows $K(x)$, while the orange region represents $\Phi(x)$ of the proposed reformulation. $\Phi(x)$ is a ball centered at $\pi_{\rm f}(x)$ with radius $r(x)$.

While the original problem~\eqref{eqn:intro_qp} is a \ac{QP}, \eqref{eqn:socp_r} is a \ac{SOCP}. Both are convex problems that can be solved efficiently~\cite{convex_opt_boyd}. However, in this case, we can also show that \ac{SOCP} has a closed-form solution. Using a shift of variables, the solution to~\eqref{eqn:socp_r} can equivalently be expressed as
    \eqn{
        \label{eq:v(x)}
        \pi_{\rm socp}(x) = \pi_{\rm f}(x) + v(x),
    }
    where $v: \Xcal \to \R^m$ is given by 
    \eqnN{
    v(x) &= \argmin{v \in \R^m} \norm{ v - v_{\rm des}(x)}^2 \ \st \ \norm{v} \leq r(x),
    }
    and $v_{\rm des}(x) = \pi_{\rm{des}}(x) - \pi_{\rm f}(x)$. 
    This leads to the  following closed-form solution for $\pi_{\rm socp}$:
    \eqn{
    \pi_{\rm socp}(x) = \pi_{\rm f}(x) + \min \left( r(x), \norm{v_{\rm des}(x)} \right)   \frac {v_{\rm des}(x)}{\norm{v_{\rm des}(x)}} .
    \label{eq:analytic_sol}
    }

We characterize~\eqref{eqn:socp_r} as our main result:
\begin{theorem}
\label{theorem:main}
Let \Cref{assum:Ab}, \ref{assum:main} hold. Then, 
\begin{enumerate}
    \item there exists a unique minimizer of the optimization problem~\eqref{eqn:socp_r} for all $x \in \Xcal$. Let $\pi_{\rm socp}: \Xcal \to \Ucal$ denote this solution. 
    \item The solution satisfies $\pi_{\rm socp}(x) \in K(x)$ for all $x \in \Xcal$. 
    \item The solution $\pi_{\rm socp}$ is Lipschitz on $\Xcal$.
\end{enumerate}
\end{theorem}

\begin{proof}[Proof]
    [First Claim] By \Cref{assum:main}, the feasible space for~\eqref{eqn:socp_r} is non-empty: $\pi_{\rm f}(x) \in \Phi(x)$.  Since the objective $\norm{u-\pi_{\rm{des}}(x)}^2$ is strongly convex in $u$, there must exist a unique minimizer~\cite[Section~4.2.1]{convex_opt_boyd}.

    [Second Claim] Consider any feasible solution $\pi \in \Phi(x)$. By~\Cref{assum:Ab}.1, we know $\norm{a_i(x)}=1$ for $i\in \Jcal$. Also, since the solution satisfies~\eqref{eqn:socp_r} for all $i \in \Jcal$, 
    \eqnN{b_i(x) - a_i(x)^\top \pi_{\rm f}(x) &\geq \norm{\pi-\pi_{\rm f}(x)} \\
    & = \norm{a_i(x)}\norm{\pi-\pi_{\rm f}(x)}\\
    & \geq a_i(x)^\top (\pi-\pi_{\rm f}(x))
    }
Therefore $b_i(x) \geq a_i(x)\pi$, i.e., $\pi \in K(x)$.

[Third Claim] Notice $v_{\rm des}$ is Lipschitz with constant $L_{v_{\rm des}} = L_{\pi_{\rm {des}}} + L_{\pi_{\rm f}}$.
    Recall $r(x) = \min_{i \in \Jcal}(r_i(x))$, where $r_i(x) = b_i(x) - a_i(x)^\top \pi_{\rm f}(x)$. Therefore, each $r_i$ has Lipschitz constant\footnote{Recall that for $f, g: X \to Y$, the Lipschtiz constant of $fg$ is $L = (\sup_{x \in X} \norm{g(x)}) L_f + (\sup_{x \in X} \norm{f(x)}) L_g $. See~\cite[Prop. 2.3.4]{lipschitz_functions_textbook}.  } $L_{r_i} = L_{b_i} + (L_{\pi_{\rm f}} +  L_{a_i} \overline{U_f})$. Therefore by~\cite[Prop. 2.3.9]{lipschitz_functions_textbook}, $r$ has the Lipschitz constant $L_{r} = \max_{i \in \Jcal}(L_{r_i})$. 
    Finally, using~\Cref{lemma:projection_1} we know that $v$ has Lipschitz constant $L_v = L_{v_{\rm des}} + L_r$. Hence $\pi_{\rm socp}$ is Lipschitz  with constant $L= L_{\pi_{\rm f}} + L_{v} = L_{\pi_{\rm des}} + 2L_{\pi_{\rm f}} + L_r$.
\end{proof}

The proof relies on~\Cref{lemma:projection_1}, provided in the appendix. 


\begin{remark}
Note that~\Cref{theorem:main} does not rely on the structural properties or conditions of the~\ac{QP}, such as constraint qualifications. 
Instead, we reformulate it as an~\ac{SOCP} in the form of~\eqref{eqn:socp_r}, which by construction is guaranteed to yield a Lipschitz-continuous solution.
\end{remark}

We can also establish local Lipschitz continuity:
\begin{corollary}
    \label{cor:locally_lip}
    Suppose \Cref{assum:Ab} holds except that $a_i, b_i, \pi_{\rm des}$ are locally Lipschitz, and that $\pi_{\rm f}: \Xcal \to \Ucal$ is locally Lipschitz and satisfies $\pi_{\rm f}(x) \in K(x)$ for all $x \in \Xcal$. Then Claims 1 and 2 of~\Cref{theorem:main} hold, and $\pi_{\rm socp}$ is locally Lipschitz continuous on $\Xcal$. 
\end{corollary}
The proof is omitted for brevity but follows that of~\Cref{theorem:main}. Notice here $u_f$ need not be bounded.\footnote{In particular, the boundedness of $\pi_{\rm f}$ is only needed to ensure $r_i$ is Lipschitz. When $a_i$ and $\pi_{\rm f}$ are both locally Lipschitz, $r_i$ is also locally Lipschitz, even when $\pi_{\rm f}$ is unbounded \cite[Prop. 2.3.4]{lipschitz_functions_textbook}.}
    

 

\subsection{Constructing the feasible solution $\pi_{\rm f}$}
\label{sec:pi_f}

Here we relax~\Cref{assum:main} by providing methods to construct $\pi_{\rm f}$. Naturally, the performance of our \ac{SOCP}~\eqref{eqn:socp_r} depends on $\pi_{\rm f}$: the farther $\pi_{\rm f}$ lies from the boundary of the feasible set $K(x)$, the larger $\Phi(x) \subset K(x)$, and consequently, the closer $\pi_{\rm socp}(x)$ can be to the desired $\pi_{\rm des}(x)$. Although $\pi_{\rm f}$ can be hand-crafted for specific problems (as in the examples above), we identify two candidates for $\pi_{\rm f}$ for general problems: the Analytic center and the Steiner point.

\textbf{Analytic Center:}
Suppose $K(x)$ as defined in~\eqref{eqn:K} is compact and has nonempty interior. Then the analytic center~\cite[Ex. 4.6]{convex_opt_boyd} of $K(x)$ is the unique solution to
\eqn{
\operatorname{Ac}(K(x)) = \argmin{u \in \R^m} \ & -\sum_{i=1}^{p} \log \left( b_i(x)-a_i(x)u \right).
\label{eq:analytic_center_}
}



This can be used to define the feasible solution $\pi_{\rm f}$:

\begin{corollary}
    Suppose (i)~\Cref{assum:Ab} holds, (ii)~$a_i, b_i \in C^2$, (iii)~$K$ is compact-valued, and (iv)~\ac{SC} holds for~\eqref{eqn:intro_qp} for all $ x \in \Xcal$. Then, if $\pi_{\rm f}(x) =\operatorname{Ac}(K(x))$, $\pi_{\rm socp}$ in~\eqref{eqn:socp_r} is locally Lipschitz continuous on $\Xcal$.
    \label{cor:main_and_analytic}
\end{corollary}
\begin{proof}
    By~\Cref{lemma:analytic_center} (see Appendix), $\pi_{\rm f}$ is locally Lipschitz on $\Xcal$. Hence the claim holds by~\Cref{cor:locally_lip}.
\end{proof}


\textbf{Steiner Point:}
The Steiner point is a Lipschitz selection of $K$ provided the set-valued map $K$ is Lipschitz in the sense of~\Cref{def:set_lipschitz}~\cite[Section 2.4.3]{freeman2008robust}. The Steiner point of a nonempty compact convex set $C \subset \R^m$ is defined by 
\eqn{
\operatorname{St}(C) = \frac{1}{\int_{B} d\mu} \int_{B} \nabla \sigma_{C} \ d\mu,
}
where $B$ is the unit ball in $\R^m$, $\mu$ is the Lebesgue measure, and $\sigma_C: \R^m \to \R$ is the support function of $C$~\cite[Eq. 2.16]{freeman2008robust}. For $K: \Xcal \toSet \R^m$ with Lipschitz constant $L_K$, the Lipschitz constant for $x \mapsto \operatorname{St}(K(x))$ is $L = m L_K$. 
\begin{corollary}
    Suppose $K: \Xcal \toSet \R^m$ is Lipschitz continuous and compact for all $x \in \Xcal$. If $\pi_{\rm f}(x) = \operatorname{St}(K(x))$, then $\pi_{\rm socp}$ in~\eqref{eqn:socp_r} is Lipschitz continuous on $\Xcal$. 
\end{corollary}

Note, even when $K$ is Lipschitz, it is known that $\pi_{\rm qp}$ as in~\eqref{eqn:intro_qp} is only H{\"o}lder continuous (see \cite{bednarczuk2021lipschitz}, \cite[Ex. 1.2]{dontchev1983perturbations}). Our proposed \ac{SOCP} yields Lipschitz continuous solutions. 


\subsection{A Conjecture}

We now propose a convex \ac{QCQP} that we conjecture is also Lipschitz continuous. Consider the following definition for $\Phi$: 
\eqn{
\Phi_{\rm qcqp}(x) = \{ u \in \R^m : h_i(x, u) \leq 0 \}
}
where for any fixed $k > 0$, 
\eqn{
h_i(x, u) = \norm{ u - \pi_{\rm f}(x)}^2 - 2 k \left( b_i(x) - a_i(x)^\top u \right). \label{eqn:h_i}
}
Compare this with $g_i$ in~\eqref{eqn:g_i}. There are three key differences: (i)~the norm is squared in the first term, (ii)~we introduced a tuning parameter $k > 0$, (iii)~the last term is $a_i(x)^\top u$, not $a_i(x)^\top \pi_{\rm f}(x)$. This changes the structure of $\Phi(x)$ from an intersection of concentric balls centered at $\pi_{\rm f}(x)$ to an intersection of non-concentric balls containing the point $\pi_{\rm f}(x)$. This design expands the feasible space and introduces a tuning parameter that allows the user to control the conservatism. By some algebra, the modified problem is 
\seqn[\label{eqn:qcqp}]{
\argmin{u \in \R^m} \ & \norm{u - \pi_{\rm des}(x)}^2\\
\st \ & \norm{u - c_i(x)}^2 \leq d_i(x),
}
where
\eqnN{
c_i(x) &= \pi_{\rm f}(x) - k a_i(x),\\
d_i(x)&= k^2 + 2 k \left( b_i(x) - a_i(x)^\top \pi_{\rm f}(x) \right).
}

\begin{conjecture}
\label{conjecture:qcqp}
    Let \Cref{assum:Ab}, \ref{assum:main} hold and $k > 0$.  Then, 
\begin{enumerate}
    \item there exists a unique minimizer of the optimization problem~\eqref{eqn:qcqp} for all $x \in \Xcal$. Let $\pi_{\rm qcqp}: \Xcal \to \Ucal$ denote this solution. 
    \item The solution satisfies $\pi_{\rm qcqp}(x) \in K(x)$ for all $x \in \Xcal$. 
    \item The solution $\pi_{\rm qcqp}$ is Lipschitz on $\Xcal$.
\end{enumerate}
\end{conjecture}
\begin{proof}[Partial Proof]
[First Claim] Notice that since $b_i(x) - a_i(x)^\top \pi_{\rm f}(x) \geq 0$ by \Cref{assum:main} and $k > 0$, we have $d_i(x) \geq k^2$. Since $\norm{a_i(x)} = 1$ by~\Cref{assum:Ab}.1,  $u = \pi_{\rm f}(x)$ is a feasible solution. Since the objective is strictly convex, there always exists a unique solution to~\eqref{eqn:qcqp}.

[Second Claim] Consider any feasible solution $u$. Then $h_i(x, u) \leq 0$. Then, $\norm{ u - \pi_{\rm f}(x)}^2 \leq 2 k (b_i(x) - a_i(x)^\top u)$, which can only be satisfied if $b_i(x) - a_i(x)^\top u \geq 0$. Since this is true for all $i \in \Jcal$, $u \in K(x)$.

[Third Claim] Proving this remains an open question. 
\end{proof}


Importantly, one can construct an example where if instead the constraints  $\tilde h_i(x, u) =  \norm{ u - \pi_{\rm f}(x)} - 2 k ( b_i(x) - a_i(x)^\top u)$ are used, the solution can be discontinuous. The square on the norm in~\eqref{eqn:h_i} appears to be important.

\section{Case Studies}
\label{sec:application}
In this section, we demonstrate our approach using our earlier examples. Recall \Cref{ex:dev_ex}. For $x \in \Xcal = [-2, 2]$, the analytic solution to the \ac{QP} is given by
\eqn{
\pi_{\rm qp}(x) = \begin{cases}
    \bmat{1 & 1}^\top & \text{if } x  \in (0, 1/3],\\
    \bmat{\frac{1 + x -2 x^2}{1 + x^2} & \frac{3x + x^2 }{1 + x^2}}^\top  & \text{else}
\end{cases}
}
which is  discontinuous at $x = 0$, as depicted in~\Cref{fig:example_1}.  Notice that this problem admits a Lipschitz solution 
\eqn{
\label{eq:feasible_solution_ex1}
\pi_{\rm f}(x) = \bmat{1 - x^2 &  1 + 2x}^\top.
}
Using this feasible solution, we can formulate a \ac{SOCP} using~\eqref{eqn:socp_r} that yields a Lipschitz continuous solution, depicted in \Cref{fig:example_1}~(d)).



Similarly, we can reformulate \Cref{ex:motivation_ex} into \acp{SOCP} using~\eqref{eqn:socp_r}. The analytic solution for \Cref{ex:motivation_ex} is 
\eqnN{
\pi_{\rm socp}(x) = \bmat{2 + \abs{x_2} -\min \left(1 + \abs{x_2}, \frac{1 - x_2 + \abs{x_2}}{\sqrt{1 + x_1^2}}\right) \\ 0}
}
which is Lipschitz on $\Xcal$. A similar but more complicated expression is available for~\Cref{ex:dev_ex} in the code repository.\footnote{\href{https://github.com/joonlee16/Lipschitz-controllers}{https://github.com/joonlee16/Lipschitz-controllers}}

Now, we consider Robinson's famous counterexample~\cite{robinson1982}:
\seqn[\label{eq:robin_raw}
]{
\argmin{u \in \R^4} \ & \frac {1}{2}\norm{ u}^2\\
\st \ & \bmat{0 & 1 & -1 & 0 \\
0 & -1 & -1 & 0 \\
1 & 0 & -1 & 0\\
-1 & 0 & -1 & -x_1} u \leq \bmat{-1 \\ -1 \\ -1 \\ -1 - x_2}
\label{eq:robinson_constraints}
}
for $x\in \Xcal=\{x\in \R^2 : \norm{x}_{\infty}\leq 2\}$. Although~\eqref{eq:robin_raw} is well-posed (satisfying \ac{SC}, strictly convex objective, etc), it was proven in~\cite{robinson1982} that the solution is not locally Lipschitz continuous at $x=\bmat{0 & 0}^\top$. It has a feasible solution $\pi_{\rm f}(x)=\bmat{0 & 0 & 2+\abs{x_2} & 0}^\top$, that is Lipschitz continuous and bounded on $\Xcal$. We can use the $\pi_{\rm f}$ above to reformulate~\eqref{eq:robin_raw} into an \ac{SOCP} form, and obtain a $\pi_{\rm socp}$ that is Lipschitz continuous by~\Cref{theorem:main}. Alternatively, we may construct the \ac{SOCP} using the analytic center~\eqref{eq:analytic_center_} which guarantees local Lipschitz continuity by~\Cref{cor:main_and_analytic}. Before transforming~\eqref{eq:robin_raw} into the \ac{SOCP} with the analytic center, note that its feasible space $K(x)=\{u \in \R^4 : A(x)u \leq b(x)\}$ is not compact, violating one of the assumptions in~\Cref{cor:main_and_analytic}. To ensure the compactness, we introduce bounding box constraints $ \norm{u}_\infty \leq u_{\rm max}$, with $u_{\max} = 5$ chosen large enough so that they are always inactive. We can now use the analytic center and solve~\eqref{eqn:socp_r}. The $\pi_{\rm socp}$ is locally Lipschitz continuous on $\Xcal$.

\section{CONCLUSIONS}

\label{sec:conc}

In this letter, we propose a new framework for optimization-based control synthesis that ensures a locally Lipschitz continuous control law. To this end, we reformulate a given \acf{QP}, a common form of optimization-based controllers, into a \acf{SOCP}. This reformulation guarantees Lipschitz continuity without requiring specific structural assumptions on the constraint matrices and admits an analytical solution form. As future work, we aim to develop an extended framework that further relaxes assumptions and verifies the conjecture. We also plan on using these results in multi-agent optimization-based controllers. 







\bibliographystyle{IEEEtran}
\bibliography{references_ll}

\appendix

\begin{lemma}
\label{lemma:analytic_center}
Let all assumptions in~\Cref{cor:main_and_analytic} hold. Then, the function $\pi_{\rm f}: \Xcal \to \R^m$ defined by $\pi_{\rm f}(x) = \operatorname{Ac}(K(x))$ is locally Lipschitz on $\Xcal$, and $\pi_{\rm f}(x) \in K(x)$. 
\end{lemma}
\begin{proof}
    Let $G(u, x) = -\sum_{i=1}^{p} \log(b_i(x)-a_i(x)u)$.
    By \ac{SC}, $K(x)$ has a non-empty interior at each $x \in \Xcal$. Hence, for $\bar x \in \Xcal$, $\bar u = \pi_{\rm f}(\bar x)$ is well-defined and $\bar u \in K(\bar x)$. Since $\log(\cdot)$ is smooth, 
    \eqnN{
    \nabla_{u} G(\bar u, \bar x) &= \sum_{i=1}^p \frac {a_i(\bar x)}{b_i(\bar x)-a_i(\bar x)^\top \bar u}=0,\\
    \nabla_{uu} G(\bar u, \bar x) &= \sum_{i=1}^p \frac{a_i(\bar x)^\top a_i(\bar x)}{(b_i(\bar x)-a_i(\bar x)^\top \bar u)^2} > 0.
    }
    Both expressions are well defined since by \ac{SC}, $\bar u$ satisfies $a_i(\bar x)^\top \bar u < b_i(\bar x)$. Similar expressions can be obtained for $\nabla_{x}G$, $\nabla_{xx}G$, $\nabla_{ux}G$. Therefore $G \in C^2$, and by the implicit function theorem~\cite[Theorem 4.2]{still2018lectures}, for any $\bar x \in \Xcal$ there exists an open neighborhood on which $\pi_{\rm f}$ is continuously differentiable. Ergo $\pi_{\rm f}$ is locally Lipschitz on $\Xcal$. 
\end{proof}

\begin{lemma}
\label{lemma:projection_1}
    Consider the optimization problem 
    \eqnN{
    v(x) = \argmin{v \in \R^m} \norm{ v - v_d(x)}^2 \ \st \  \norm{v} \leq r(x)
    }
    where $v_d: \Xcal \to \R^m$, $r: \Xcal \to \Rnonneg$ are Lipschitz continuous with constants $L_{v_d}, L_r \geq 0$ respectively. Then $v: \Xcal \to \R^m$ is also Lipschitz continuous with constant $L_{v_d} + L_r$. 
\end{lemma}
\begin{proof}
    For convenience, define $n(x) = \frac{v_d(x)}{\norm{v_d(x)}}$, with the convention that if $\norm{v_d(x)} = 0, n(x) = 0$.  
    Notice $v$ can be expressed as
    \eqn{
        v(x) &= \operatorname{Proj}_{r(x)}(v_d(x))
        =\min( r(x), \norm{v_d(x)}) n(x) \label{eqn:vx}
    }
    where $\operatorname{Proj}_r(u)$ denotes the projection of $u \in \R^n$ onto the closed ball of radius $r \geq 0$ centered at the origin. 
    
    We aim to prove that there exists a $L\geq 0$ such that
    \eqnN{
    \norm{v(x) - v(y)} \leq L \norm{x - y}
    }
    for all $x, y \in \Xcal$. Consider three cases:
    
    \textbf{Case 1} [$r(x) = r(y) = 0$]: 
    
    Since $v(x) = v(y) = 0$, the claim holds with $L = 0$. 

    \textbf{Case 2} [$r(x) = 0,  r(y) > 0$]:
    
    Since $v(x) = 0$, and $v(y)$ is as in~\eqref{eqn:vx},
    \eqnN{
    \norm{v(x) - v(y)} &= \norm{\min( r(y), \norm{v_d(y)}) \  n(y)}\\
    &= \min( r(y), \norm{v_d(y)}) \norm{ n(y)}\\
    &\leq \min( r(y), \norm{v_d(y)})\\
    &\leq r(y) = \abs{ r(x) - r(y)}
    }
    Hence the claim is true with $L = L_r$. 
    
    \textbf{Case 3} [$r(x),  r(y) > 0$]:
    Consider two subcases. 
    
    \textbf{Case 3a} [$\norm{v_d(x)} \leq r(x), \norm{v_d(y)} \leq r(y)$]:
    
    Here, $v(x) = v_d(x), v(y) = v_d(y)$, and hence
    \eqnN{
    \norm{v(x) - v(y)} &= \norm{ v_d(x) - v_d(y)} \leq L_{v_d} \norm{ x - y}
    }
    thus, the claim holds with $L = L_{v_d}$. 

    \textbf{Case 3b} [$\norm{v_d(y)} > r(y)$]:
    Let $p = \operatorname{Proj}_{r(x)}(v_d(y)) = r(x) n(y)$. Then,
    \eqnN{
    &\norm{v(x) - v(y)} \leq \norm{ v(x) - p} + \norm{ p - v(y)}\\
    &= \norm{ \operatorname{Proj}_{r(x)} v_d(x) - \operatorname{Proj}_{r(x)} v_d(y)} + \norm{r(x)n(y) - r(y)n(y)}.
    }
    The first term on the right is bounded by $\norm{v_d(x) - v_d(y)}$ since the projection onto a closed convex set has Lipschitz constant 1. The second term is bounded by $\abs{ r(x) - r(y)} \norm{ n(y)}$. Since $\norm{n(y)} \leq 1$, 
    \eqnN{
    \norm{v(x) - v(y)} &\leq \norm{v_d(x) - v_d(y)} + \abs{ r(x) - r(y)} \norm{ n(y)}\\
     &\leq L_{v_d}\norm{ x - y} + L_r \norm{x-y}.
    }
    Hence the claim is true with $L = L_{v_d} + L_r$. 
\end{proof}

\end{document}